\documentclass[10pt,a4paper,oneside,final]{article}
\usepackage{amsmath,amsthm,amssymb,graphicx}
\usepackage{lineno}
\newtheorem{thm}{Theorem}[section]

\newtheorem{rem}[thm]{Remark}
\newtheorem{prop}[thm]{Proposition}
\theoremstyle{definition}
\newtheorem{defn}[thm]{Definition}
\newcommand{\Field}[1]{\mathbb{#1}}

\newcommand{\Z}{\Field{Z}}

\begin{document}

\title{%        %You can use \\ for explicit line-break.
A study of the entanglement in systems with periodic boundary conditions%
}

%\subtitle{Subtitle}    %Use this when you want a subtitle.

\author{%       %Use \scshape for the family name.
 E Panagiotou $\flat$, C Tzoumanekas $\sharp$, S Lambropoulou $\flat$ ,K C Millett $\natural$\\
 and D N Theodorou $\sharp$}
\date{}
\maketitle

\noindent$\flat$ Department of Mathematics, National Technical University, Athens, GR 15780, Greece\\
$\sharp$ Department of Materials Science and Engineering, School of Chemical
Engineering, National Technical University, Athens, GR 15780, Greece\\
$\natural$ Department of Mathematics, University of California, Santa Barbara, California 93106\\

\begin{abstract}
We define the local periodic linking number, $LK$, between two oriented closed or open chains in a system with three-dimensional periodic boundary conditions. The properties of $LK$ indicate that it is an appropriate measure of entanglement between a collection of chains in a periodic system.
Using this measure of linking to assess the extent of entanglement in a polymer melt we study the effect of CReTA algorithm on the entanglement of polyethylene chains.
Our numerical results show that the statistics of the local periodic linking number observed for polymer melts before and after the application of CReTA are the same.
\end{abstract}

\section{Introduction}
Polymer melts are dense systems of macromolecules.
In such dense systems the conformational freedom and motion of each chain is significantly affected by entanglement with other chains which generates obstacles of topological origin to its movement \cite{Ed,PolPhys}. These obstacles are called \textit{topological constraints} (TCs) (Figure \ref{Theod}). Edwards suggested that the local interaction between the chains restricts their motion to a tube-like region. The axis of the tube is a coarse-grained representation of the chain and it is called the \textit{primitive path} \cite{Ed}. Owing to its simplicity, the tube model allows us to carry out a theoretical analysis and describe the dynamical and rheological properties of polymer melts\cite{Laso,kazu,krog1,PolPhys,Cossms}.
The CReTA (Contour Reduction Topological Analysis) algorithm \cite{TopAn} employs the Doi-Edwards perspective to study the entanglement in polymer melts. It is a coarse graining algorithm that reduces a computer generated atomistic sample to a simpler entanglement network of primitive paths (Figure \ref{Theod}), where inside the vertices occur TCs.

\begin{figure}
   \begin{center}
     \includegraphics[width=0.25\textwidth]{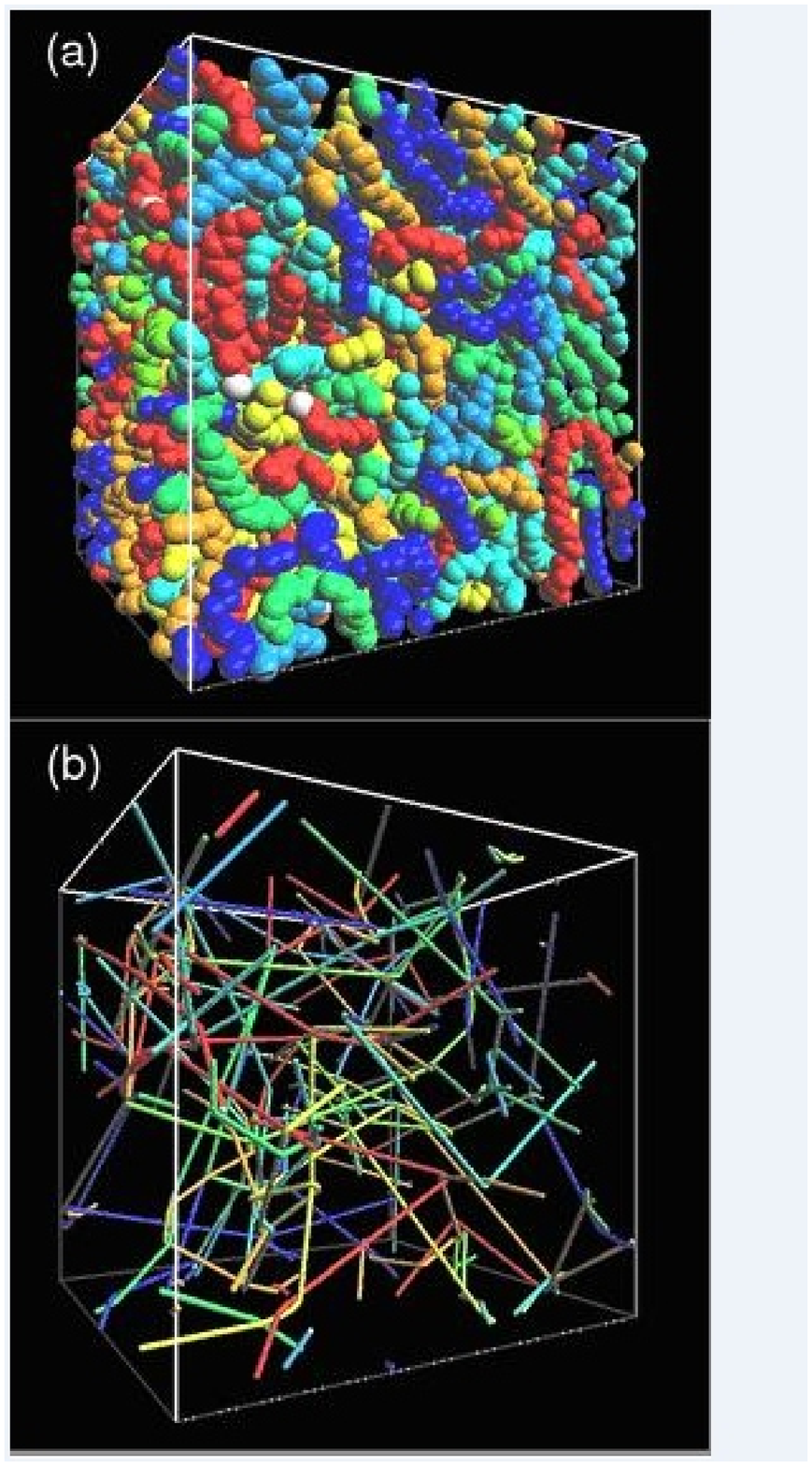}\includegraphics[width=0.6\textwidth]{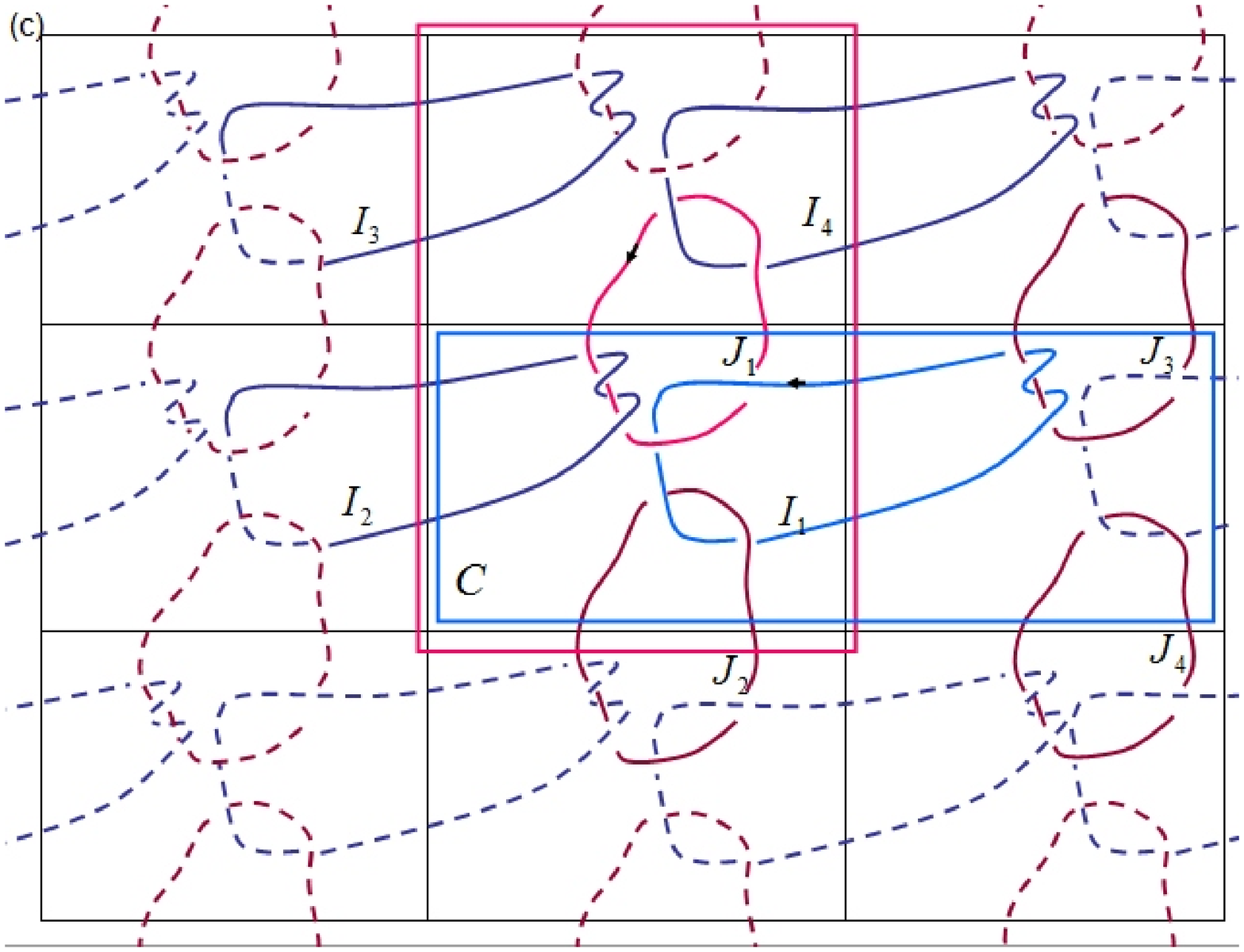}
     \caption{(a) Representative atomistic PE sample and (b) the corresponding reduced network. (c) The central cell $C$ and the periodic system it generates. The generating chain $i$ (resp. $j$) is composed by the blue (resp. red) arcs in $C$. The free chain $I$ (resp. $J$) is the set of blue (resp. red) chains in the periodic system. Highlighted are the minimal unfoldings of the images $I_1$ and $J_1$. For the computation of $LK(J,I)$ we have $LK(J,I)=L(J_1,I_1)+L(J_1,I_2)+L(J_1,I_3)+L(J_1,I_4)$.
We can see that $L(I_1,J_1)=L(J_1,I_1),L(I_1,J_2)=L(J_1,I_4),L(I_1,J_3)=L(J_1,I_2)$ and $L(I_1,J_4)=L(J_1,I_3)$, thus $LK(I,J)=LK(J,I)$.}
     \label{Theod}
   \end{center}
\end{figure}

The computer simulation of a polymer melt usually involves molecular dynamics and Monte Carlo simulations \cite{Th,Theod,Th3}.
In order to eliminate boundary effects, for the modeling of a bulk system Periodic Boundary Conditions (PBC) are applied \cite{TheodPBC}. In a PBC model, the \textit{cubical simulation box} or \textit{cell} is replicated throughout space to form an infinite lattice. In the course of the simulation, when a molecule moves in the simulation box, its periodic image in each one of the replication boxes moves in exactly the same way. Thus, as a molecule leaves the simulation box, one of its images will enter through the opposite face at exactly the same point.
The configuration of the molecules inside the simulation box generates the entanglement present in the continuum. Because of this, the periodic system induces special features in the entanglement of the simulated polymer melt.
For each chain in the melt we can distinguish between two types of topological constraints imposed on it by the other chains: local topological constraints, which can be observed inside a simulation box, and large scale topological constraints, which are observed only at a larger length.

The purpose of this paper is to describe methods by which one may quantify and extract local entanglement information in a system with PBC. More precisely, in Section 2 we describe the structure of the PBC polymer melt model. In Section 3 we introduce the periodic linking number and the local periodic linking number as measures of entanglement in such a system with PBC and we study their properties. In Section 4 we apply the local periodic linking number to polyethylene (PE) melts and study the effect of the CReTA algorithm on the linking present in the system.

\section{Periodic Boundary Conditions}

We study a polymer melt consisting of a collection of polymer chains of length $m$, by dividing space into a family of cubic boxes of volume $L^3$, where $L$ is the length of an edge of the cube, so that the structure of the melt in each cube is identical, i.e. we impose PBC on the melt \cite{PolPhys}. Specifically, we make the following definition :

\begin{defn} A \textit{cell} consists of a cube with $n$ arcs embedded into it such that arcs may terminate only in the interior of the cube or on a face, but not on an edge or corner, and those arcs which meet a face satisfy the PBC requirement. That is, to each ending point corresponds a starting point at exactly the same position on the opposite face of the cube.
See Figure \ref{Theod}(b) for an illustrative example.
\end{defn}

A cell generates a \textit{periodic system} in 3-space by tiling 3-space with the cubes so that they fill space and only intersect on their faces. This allows an arc in one cube to be continued across a face into an adjacent cube and so on. Since physical chains are of finite length, we require that resulting chains must be compact.
For practical purposes, only a finite number of copies of the cell are considered. We call such a collection a \textit{finite periodic system}.

Without loss of generality, we choose a cell of the periodic system that we call \textit{generating cell}. Then any other cell $c$ in the periodic system is a translation of the generating cell by a vector $\vec{c}=(c_x,c_y,c_z), c_x,c_y,c_z\in L \Z$. A \textit{generating chain} is the union of all the segments inside the cell the translations of which define a maximal connected arc in the periodic system. For each arc of a generating chain we choose an orientation such that the translations of all the arcs would define an oriented arc in the periodic system. For each generating chain we choose without loss of generality an arc and a point on it to be its \textit{base point} in the central cell. The smallest union of the copies of the cell needed for one complete unfolding of a generating chain shall be called the \textit{minimal unfolding}. The smaller number of copies of the cell whose union contains the convex hull of the complete unfolding of a generating chain shall be called the \textit{minimal topological cell}.

For generating chains we shall use the symbols $i,j,\dots$. A generating chain is said to be closed when the connected component that its segments construct in the periodic system is a closed chain.
The collection of all translations of the same generating chain $i$ shall be called \textit{free chain}, denoted $I$. A free chain is identified with the collection of its connected components, each of which is a translation of an unfolding of a generating chain in the periodic system. For free chains we will use the symbols $I,J,\dots$. An \textit{image} of a free chain is any connected component in it. For images of a free chain, say $I$, we will use the symbols $I_1,I_2,\dots$. A free chain is said to be closed, when its components are closed curves.

\begin{prop}\label{symm3}
Given a generating chain $i$ in a cell, and its corresponding free chain $I$ in the periodic system, every image of $I$ contains a translation of each arc of $i$ exactly once.
\end{prop}

\begin{proof}[Sketch of Proof]
Let $i_m', i_{m+1}',\dotsc,i_{m+l}'$ denote a sequence of arcs that we meet along $I_1$. Each one of these arcs is a translation of an arc of $i$. Suppose that $i_{m+l}'=i_m'$. Then it must be that $i_{m+l+1}'=i_{m+1}'$, giving $i_{m+2l}'=i_m'$. Thus $I_1$ not compact, contradiction.

\end{proof}

Proposition \ref{symm3} shows that each image contains exactly one base point. We call the image of $I$ whose base point lies in the generating cell the \textit{parent image} and any other image of $I$ can be defined by a vector with respect to the base point of the parent image.

The geometric and topological information resulting from a periodic system is carried in its generating cell, but the global entanglement structure seen in the periodic system may be different from that seen inside a cell.
\textit{Local topological constraints (local TCs)} are those tight conformations of the arcs inside a cell that form barriers to small movements of a chain since a chain is prohibited from passing through itself or any other chain inside the cell by motions keeping the endpoints of the arcs of the generating chains fixed.

\section{Linking in PBC systems}

The Gauss linking number is a traditional measure of the algebraic entanglement of two disjoint oriented circles that extends directly to disjoint oriented open chains \cite{ACN,Uni,Ed}.

\begin{defn}[Gauss 1877]\label{lk}The \textit{Gauss linking number} of two disjoint oriented curves  $l_1$ and $l_2$, whose arc-length parametrization is $\gamma_1(t),\gamma_2(s)$ respectively, is defined as a double integral over $l_1$ and $l_2$:

\begin{equation}
L(l_1,l_2)=\frac{1}{4\pi}\int_{[0,1]}\int_{[0,1]}\frac{(\dot\gamma_1(t),\dot\gamma_2(s),\gamma_1(t)-\gamma_2(s))}{\left|\gamma_1(t)-\gamma_2(s)\right|^3}dt ds
\end{equation}

\noindent where $(\dot\gamma_1(t),\dot\gamma_2(s),\gamma_1(t)-\gamma_2(s))$
is the triple product of $\dot\gamma_1(t),\dot\gamma_2(s)$ and
$\gamma_1(t)-\gamma_2(s)$.
\end{defn}

\bigskip

In the case of closed chains the Gauss linking number is a topological invariant. If it is equal to zero, the two chains are said to be algebraically unlinked.
For open chains the Gauss linking number is a continuous function in the space of configurations and as the endpoints of the chains tend to coincide, it tends to the linking number of the resulting closed chains \cite{Uni}.
For open chains, the Gauss linking number may not be zero, even for chains whose topological cells do not intersect. But as the distance between them increases, the Gauss linking number tends to zero \cite{hira}. We note that when one considers open chains, the concept of topological linking does not apply, as open chains are always topologically unlinked, i.e. they may be deformed so as to lie in disjoint topological cells.

In the case of a periodic system a different measure of entanglement is needed in order to capture all the TCs and the effect of the periodicity of the conformations. Here, we apply the Gauss linking number and its extension to open chains to the situation of chains defined in a PBC context providing a measure of the large scale entanglement between the chains. We propose the following definition of the linking number of two free chains in a system with PBC:

\begin{defn}[Periodic linking number]\label{lk3} Let $I$ and $J$ denote two free chains in a periodic system. Suppose that $I_1$ is an image of the free chain $I$ in the periodic system. The \textit{periodic linking number}, $LK_P$, between two free chains $I$ and $J$ is defined as:

\begin{equation}\label{lk4}
LK_P(I,J)=\sum_{u}L(I_1,J_u)
\end{equation}

\noindent where the sum is taken over all the images $J_u$ of the free chain $J$ in the periodic system.

\end{defn}

For the purposes of applications to polymer melts, where the density is relatively high, the topological constraints have a strongly local character and one may wish to focus attention on the local TCs \cite{Ralph,TopAn,Onset12,kazu}. As a consequence, the following local periodic linking measure may provide a more relevant measure of the entanglement observed in a polymer melt.

\begin{defn}[Local periodic linking number]\label{lk2} Let $I$ and $J$ denote two free chains in the periodic system. Let $J_1,J_2,\dotsc,J_n$ denote the images of $J$ that intersect the minimal unfolding of an image of $I$, say $I_1$. The \textit{local periodic linking number}, $LK$, between two free chains $I$ and $J$ is defined as:

\begin{equation}\label{lk1}
LK(I,J)=\sum_{1\leq u\leq n}L(I_1,J_u)
\end{equation}

\end{defn}

\smallskip

\noindent\textbf{Example 3:} Consider the cell as shown in Figure \ref{Theod}c. Then for the computation of $LK(I,J)$ we need to consider the images of $J$ which intersect the minimal unfolding of $I_1$, as shown in Figure \ref{Theod}c: $LK(I,J)=L(I_1,J_1)+L(I_1,J_2)+L(I_1,J_3)+L(I_1,J_4)$.

\bigskip

We note that if the chains do not touch the faces of the cell, and thus lie entirely inside the cell, $LK_P$ and $LK$ are equal to the Gauss linking number of the two chains.
\noindent We shall focus on $LK$. In a sequel paper we will study $LK_P$ and its relation to $LK$.

\begin{figure}
   \begin{center}
     \includegraphics[width=0.3\textwidth]{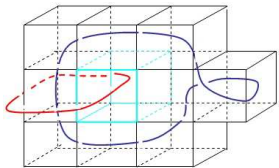}
     \caption{An example for which $LK\neq LK_P$. The red curve does not intersect the minimal unfolding of the blue curve, but they are topologically linked. If the blue chain moves into the blue cell, then the red curve will intersect its minimal unfolding and $LK=LK_P$.}
     \label{counter}
   \end{center}
\end{figure}

\begin{rem}\label{crem}\rm The local periodic linking number, $LK$, depends upon the size of the cell. Also, as the minimal unfolding may change under the motion of $I$, then $LK$ may change as well. Thus $LK$ is not a topological invariant for two closed free chains.
The minimal unfolding of an image, say $I_1$, is homeomorphic to a 3-manifold, say $M$. If the genus of $M$ is greater than $0$, there may be an image of $J$ that is linked with $I_1$ but that does not intersect the minimal unfolding of $I_1$ (see Figure \ref{counter}).
We note though, that the size of a simulation cube that is appropriate for the simulation of polymer melts and other polymer systems, is such that it is impossible for conformations such as the one shown in Figure \ref{counter} to arise.
So, for all practical purposes, the local periodic linking number is an invariant under the isotopic motion of closed chains.
\end{rem}

\subsection{Consequences and properties of the local periodic linking number}

\begin{prop}\label{sym} The local periodic linking number $LK$ between two free chains $I$ and $J$ of a system with PBC is symmetric, that is $LK(I,J)=LK(J,I)$.
\end{prop}

\begin{proof}

Let us consider an image of $J$, say $J_1$, that intersects the minimal unfolding of an image of $I$, say $I_1$. Then every other image of $J$ that intersects the minimal unfolding of $I_1$ is a translation of $J_1$ by a vector $\vec{v_u}$. Then we have that $I_1$ intersects the minimal unfolding of each $J_u=J_1+\vec{v_u}$. Thus the translation of $I_1$ by $-\vec{v_u}$ , i.e. $I_u=I_1-\vec{v_u}$, intersects the minimal unfolding of $J_1$.

\end{proof}

The local periodic linking number is a continuous function. Indeed, the set of points of discontinuity in the space of possible configurations is of measure zero, so $LK$ is a continuous function of the chain coordinates almost everywhere.
The following properties follow from the fact that $LK$ is a measure of entanglement.

\noindent (i) In the case of closed chains $LK$ is an integer. Indeed, it is equal to half the algebraic number of crossings between a projection of an image of $I$, say $I_1$, and the projection of the images of $J$ which intersect the minimal unfolding of $I_1$ in any generic projection direction. For closed chains whose minimal unfolding has genus zero, then $LK$ is invariant under movement of the chain.

\noindent (ii) In the case of open chains, $LK$ is equal to the average algebraic number of crossings between a projection of an image of $I$, say $I_1$, and the projection of the images of $J$ that intersect the minimal unfolding of $I_1$. Moreover, as the endpoints of the generating chains move closer, $LK$ tends to the local linking periodic number of the closed free chains.

Further, the local periodic linking number has the following properties with respect to the structure of the cell:

\noindent (i) $LK$ captures the contributions of all the local TCs that the one free chain imposes on an image of the other, but it also takes into account the global geometrical structure of the chains.

\noindent (ii) $LK$ is independent of the choice of the cell and of the translated image of a free chain $I$ that we use for the computation of $LK(I,J)$ with another free chain $J$.

\noindent (iii) For closed chains in a periodic system, one may amalgamate cells to create larger generating cells in a new PBC such that $LK=LK_P$, and then $LK$ is a topological invariant.

\subsection{Advantages of the local periodic linking number for the study of the linking in PBC}

Systems which employ PBC consist of an infinite number of translations of a finite number of chains. Applying a traditional measure of entanglement would imply computations involving an infinite number or, at least, a very large number of chains. Furthermore, because of the periodicity of the system, there is a scale at which one has homogeneity, so that entanglement can be characterized by the local structure and the degree of its homogeneity.

The local periodic linking number, $LK$, is computed using a finite number of cells and, therefore, a finite number of chains thereby reducing the complexity of its calculation.

Ideally, one would like to compute a linking measure directly from one cell, but the arcs of the generating chains inside the cell are relatively short. In order to capture the greater degrees of entanglement, or even complex knotting, a large number of arcs must be employed in the creation of a complex chain.

The local periodic linking number is computed between long polymer chains, taking into consideration the entire geometrical structure of the chains.
Moreover, if one wanted to constrain the analysis to a smaller number of cells, it would be impossible to give a definition of linking that is symmetric and captures all the entanglement present in a cell. Indeed:

\begin{prop}\label{symm2}
The local periodic linking number $LK(I,J)$ between two closed or open free chains $I$ and $J$ takes into consideration as few as possible cells required to assess all the local TCs  that the one free chain imposes on an image of the other and at the same time to be symmetric.
\end{prop}

\begin{proof}[Sketch of Proof]

The minimal unfolding (\textit{mu}) of $I_1$ contains all the local TCs imposed to $I_1$ by $J$. But we have that $LK_{mu}(I,J)=\sum_l L(I_1,j_l)\neq \sum_n L(J_1,i_n)$ $=LK_{mu}(J,I)$, where $j_l,i_n$ are the translated arcs of $j$ and $i$ in the minimal unfolding of $I_1$ and $J_1$ respectively. By adding the missing terms in each summation until we get a symmetric definition of linking we obtain $LK_{mu}=LK$.
\end{proof}

\section{Application to polymers}

\subsection{The CReTA algorithm}

To study the entanglement complexity present in a given polymer system, one has to coarse grain the polymer chains
at a level where certain geometrical characteristics relevant to entanglement become evident \cite{Fot}.
The CReTA (Contour Reduction Topological Analysis) \cite{TopAn} algorithm fixes chain ends in space, and
by prohibiting chain crossing, it minimizes (shrinks) simultaneously
the contour lengths of all chains, until they become
sets of rectilinear strands coming together at the nodal points
(TCs) of a network of primitive paths.
When there are no possible alignments left, then we shrink the diameter of the beads of each chain and continue the same process, until a minimum thickness is achieved and no alignment moves are possible.
Figure \ref{Theod} shows an atomistic polymer sample and the corresponding reduced network.

\subsection{Numerical Results}

The CReTA algorithm has been shown \cite{TopAn,Onset12} to give a meaningful representation of the underlying topology of a polymer melt by comparing it to experimental data.
Using the local periodic linking number $LK$, we determine the extent to which the CReTA method preserves the linking measure of entanglement in PE melts, or if critical entanglement information is lost in the reduction process. Also, we perform the same comparison after end-to-end closure of the chains of a PE melt and its CReTA reduced melt.

As we mentioned earlier, linear polymer chains are not linked or knotted in the topological sense. A method to study the entanglement of linear polymers is to perform closure of the chains and compute the resulting knot type \cite{hanse,LinearTying}. Using $LK$ we will test if the method of direct end-to-end closure \cite{Laso,hanse,virnau}is reliable and if $LK$ can detect topological differences between the open and closed original and reduced systems.

The data analyzed concerns 80 PE frames~\cite{TopAn} of density $\rho\approx 0.77-0.78 g/cm^3$. Each frame is generated by 8 PE chains of 1000 beads each, so finally our results on $LK$ are based on 2240 distinct pairs of PE chains.

\begin{figure}
   \begin{center}
     \includegraphics[width=0.4\textwidth]{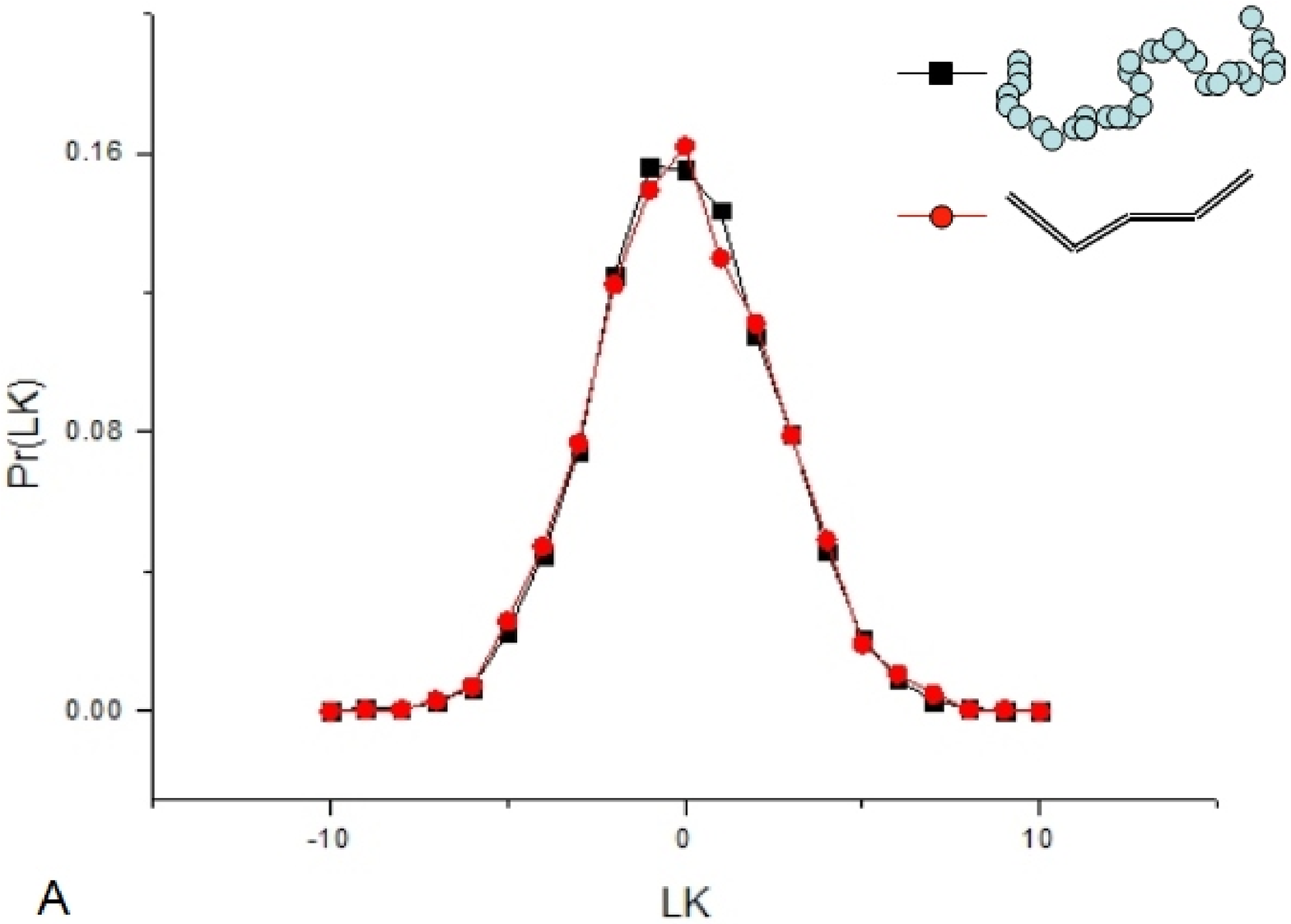}\includegraphics[width=0.4\textwidth]{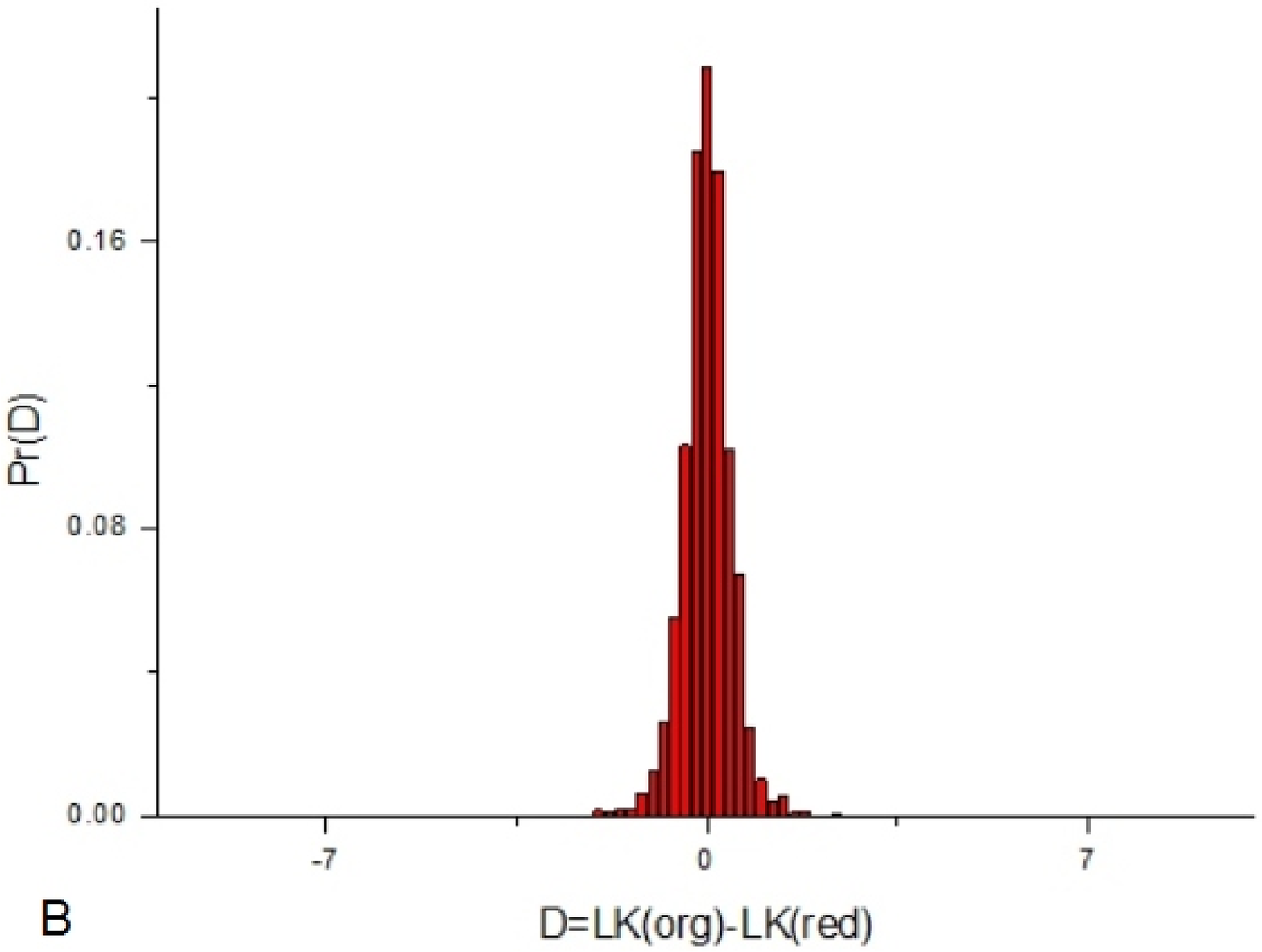}
     %\caption{(a) Representative atomistic PE sample and (b) the corresponding reduced network. This figure was kindly provided by Prof. Theodorou and Dr. Tzoumanekas \cite{TopAn}}
     \label{OORO}
   \end{center}
\end{figure}

\begin{figure}
   \begin{center}
     \includegraphics[width=0.4\textwidth]{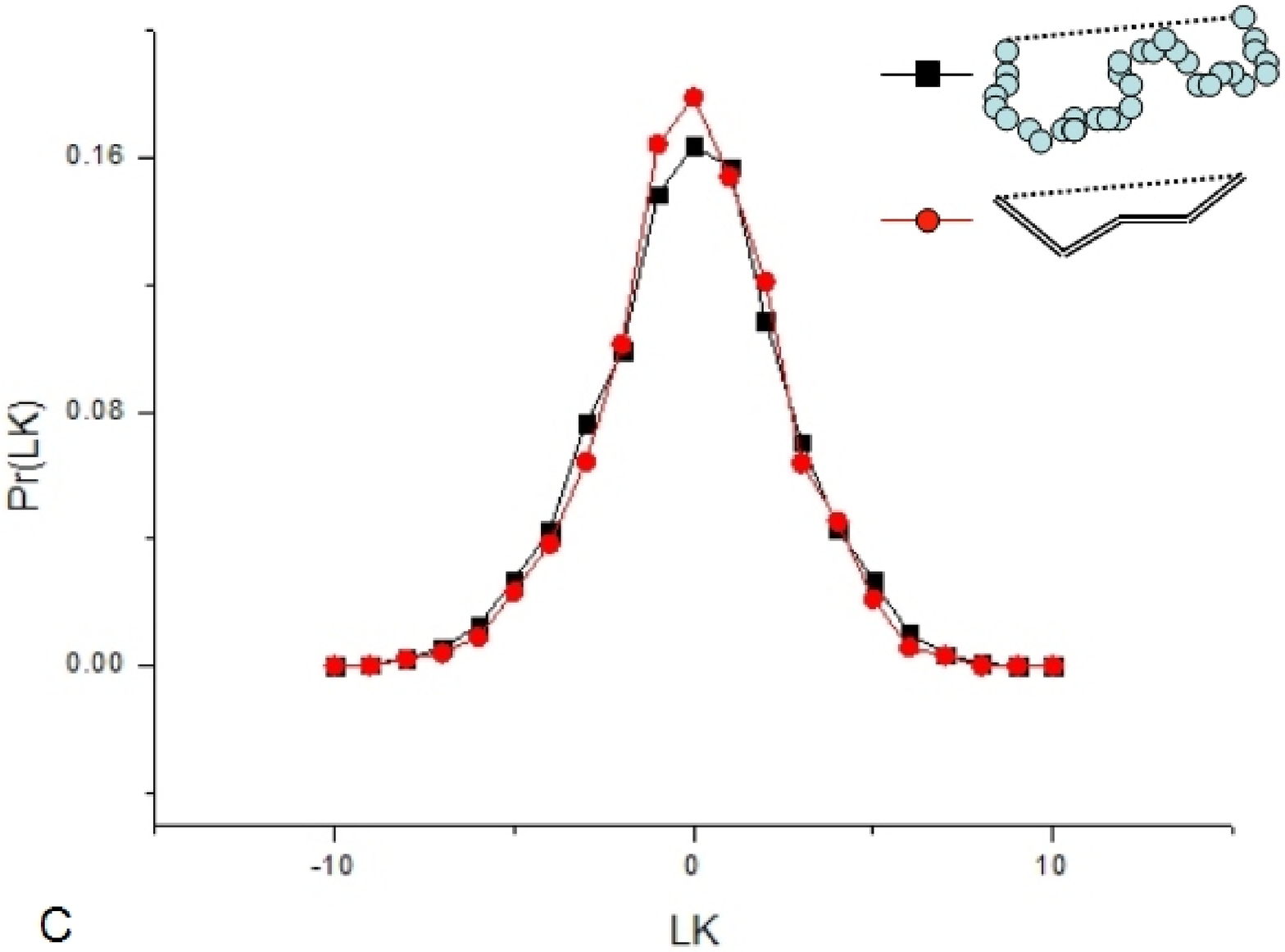}\includegraphics[width=0.4\textwidth]{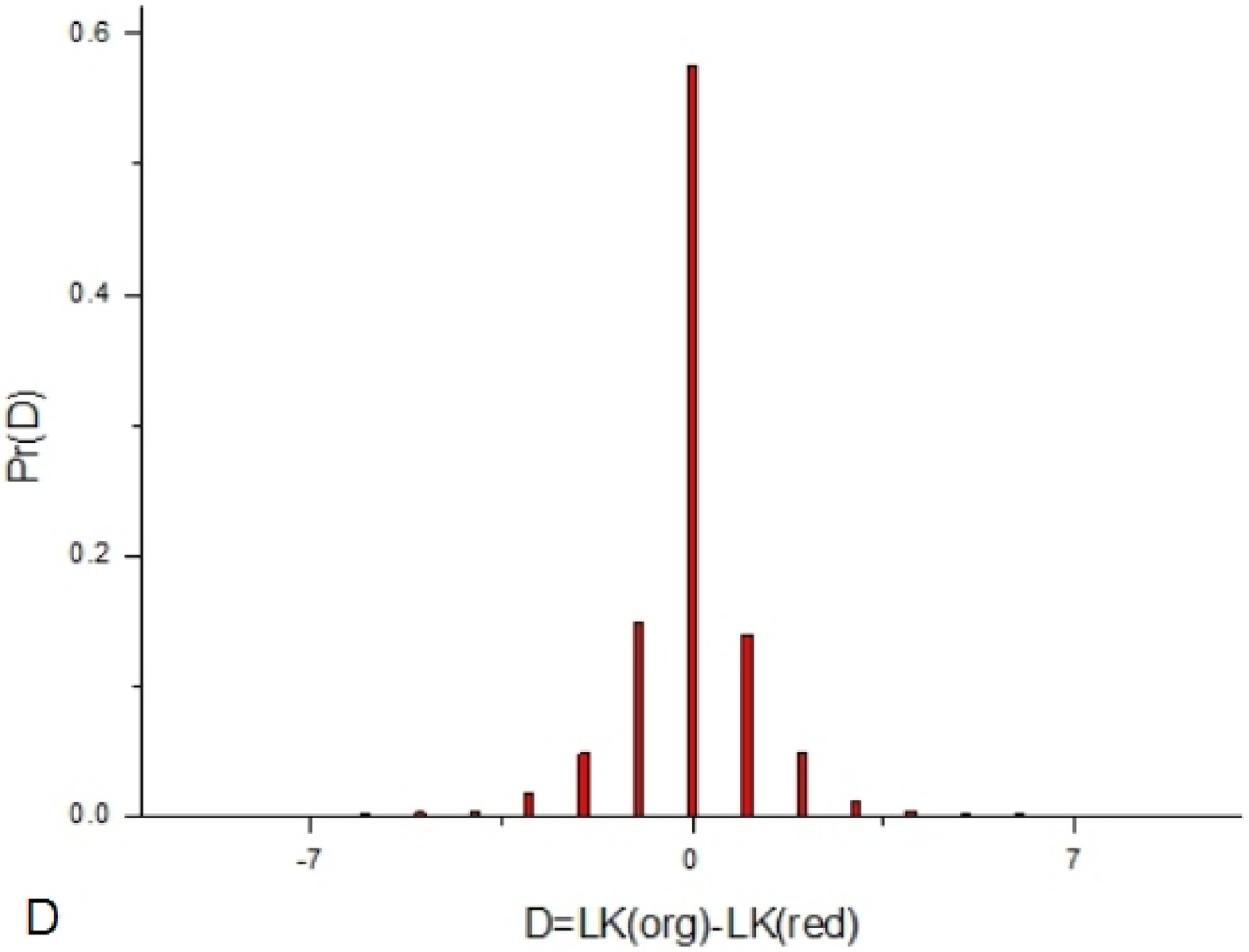}
     \caption{Figure A shows the normalized probability distribution of $LK$ for the original and the corresponding reduced chains in a PE-1000 melt and Figure B the probability distribution of the corresponding differences. Figure C shows the normalized probability distribution of $LK$ for the original and reduced chains after end-to-end closure and Figure D shows the corresponding differences}
     \label{OCRC}
   \end{center}
\end{figure}

First, we compute the normalized probability distribution of $LK$ for pairs of PE chains in a frame and compare it to the $LK$ for the corresponding reduced pairs.
Figure \ref{OCRC}A shows the probability distribution of the pairwise local periodic linking numbers for the original and CReTA reduced PE chains. The two distributions are quite similar. The mean absolute value of $LK$ is $1.97$ and the mean absolute value of $LK$ of pairs of reduced chains is $2.01$.
Figure \ref{OCRC}B shows the probability distribution of the difference of $LK$ before and after the application of the CReTA algorithm, i.e. $D=LK_{original}-LK_{reduced}$. We note that the distribution is quite narrow (with standard deviation $0.30$), has mean $-0.19$,
and mean absolute difference (i.e. $<|D|>$) of the local periodic linking between original and reduced chains of $0.29$,
indicating that \textit{$LK$ is about the same for original and reduced pairs of chains}. Thus, at least with respect to $LK$, one can study the entanglement of a PE melt in its reduced representation after the application of the CReTA algorithm.

\begin{figure}
   \begin{center}
     \includegraphics[width=0.4\textwidth]{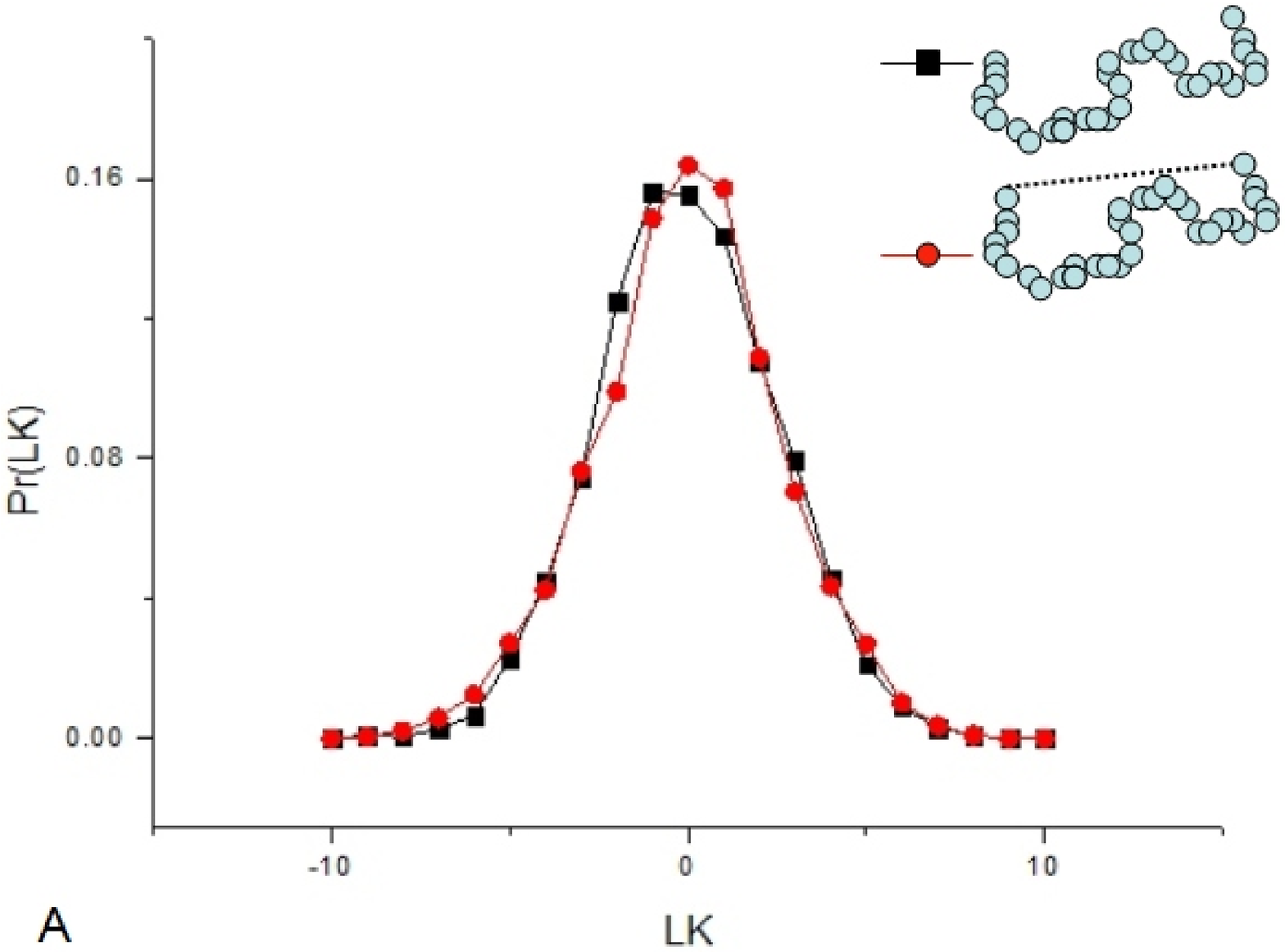}\includegraphics[width=0.4\textwidth]{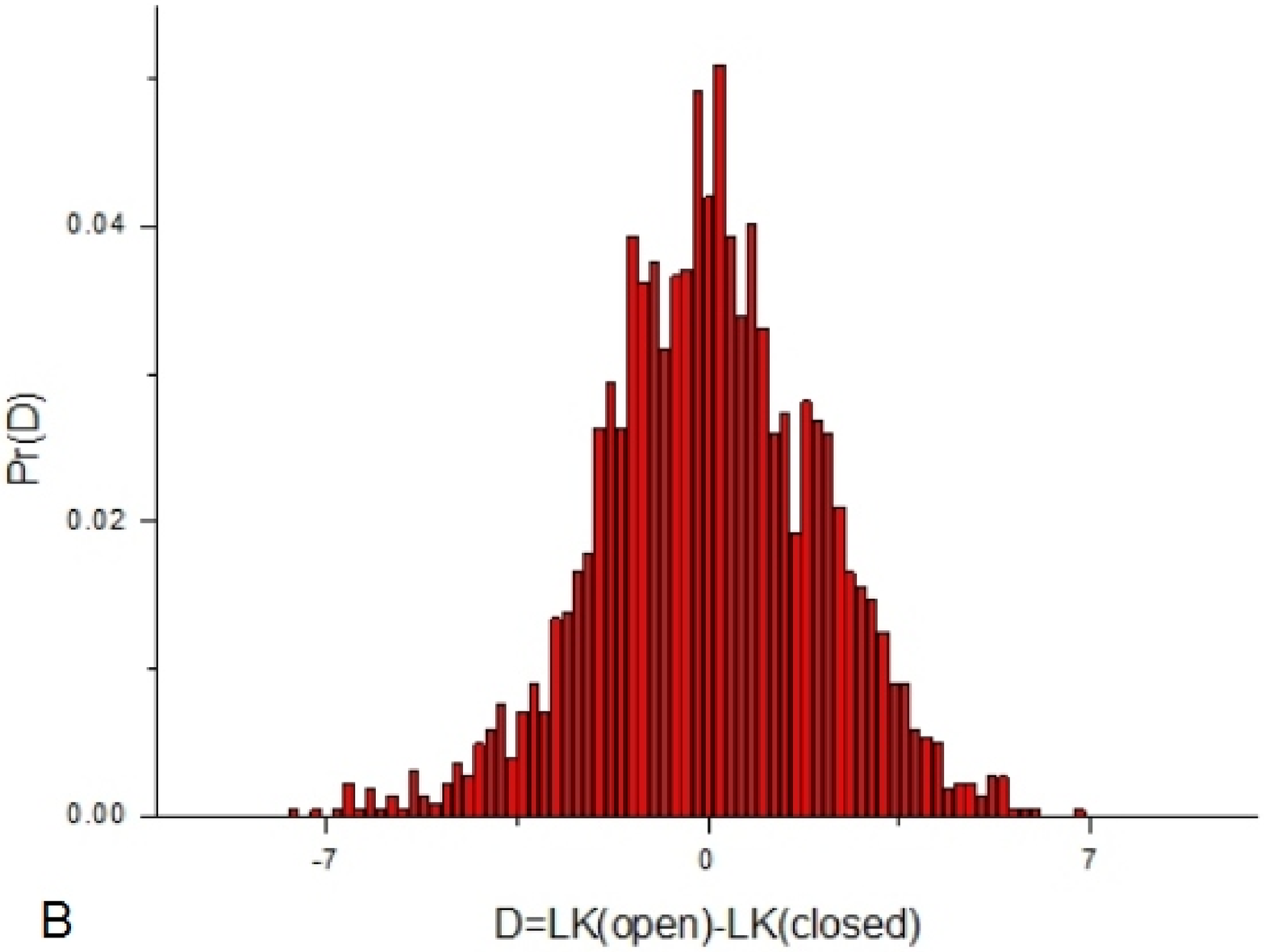}
    % \caption{(a) Representative atomistic PE sample and (b) the corresponding reduced network. This figure was kindly provided by Prof. Theodorou and Dr. Tzoumanekas \cite{TopAn}}
     %\label{Theod}
   \end{center}
\end{figure}

\begin{figure}
   \begin{center}
     \includegraphics[width=0.4\textwidth]{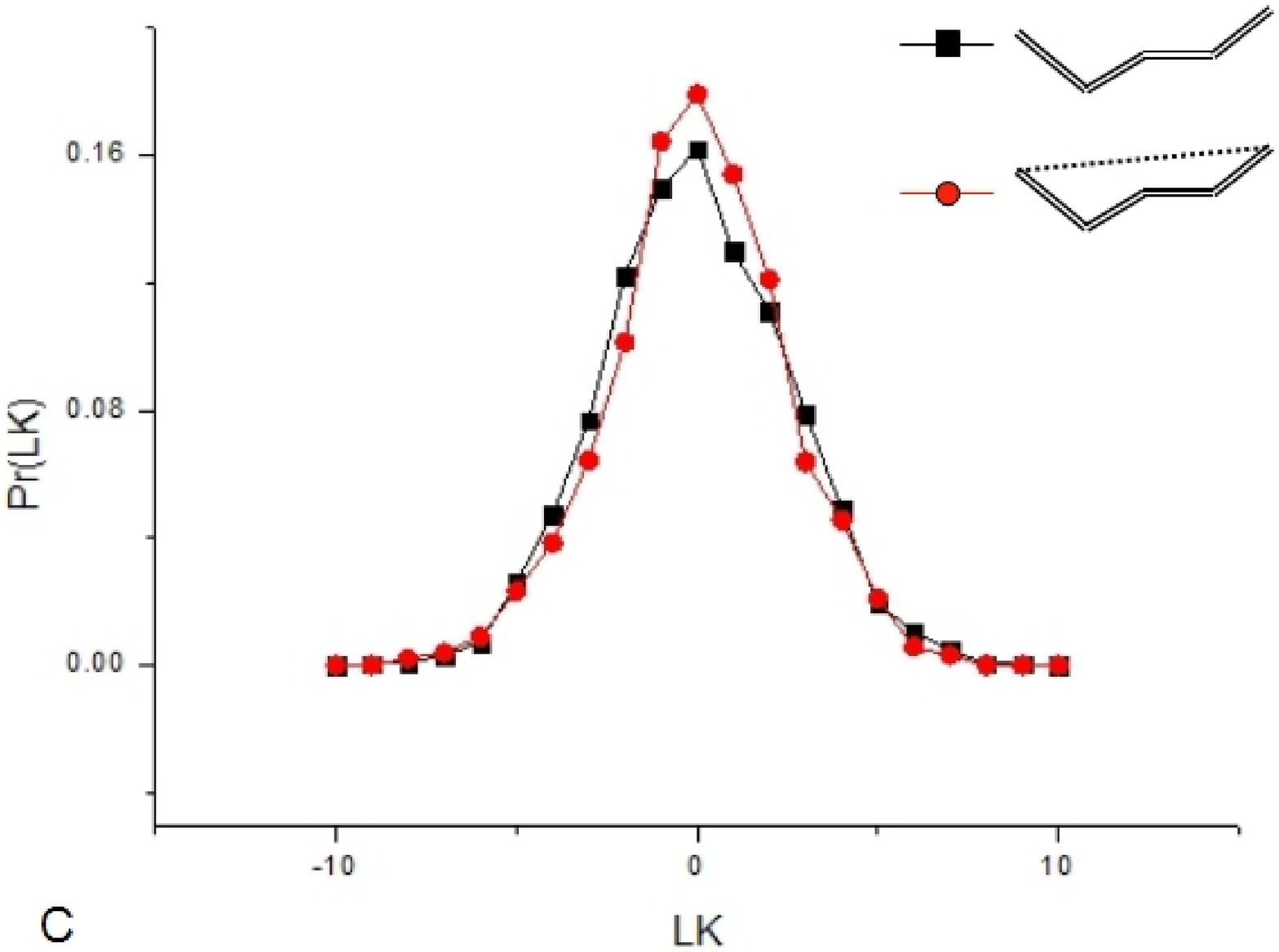}\includegraphics[width=0.4\textwidth]{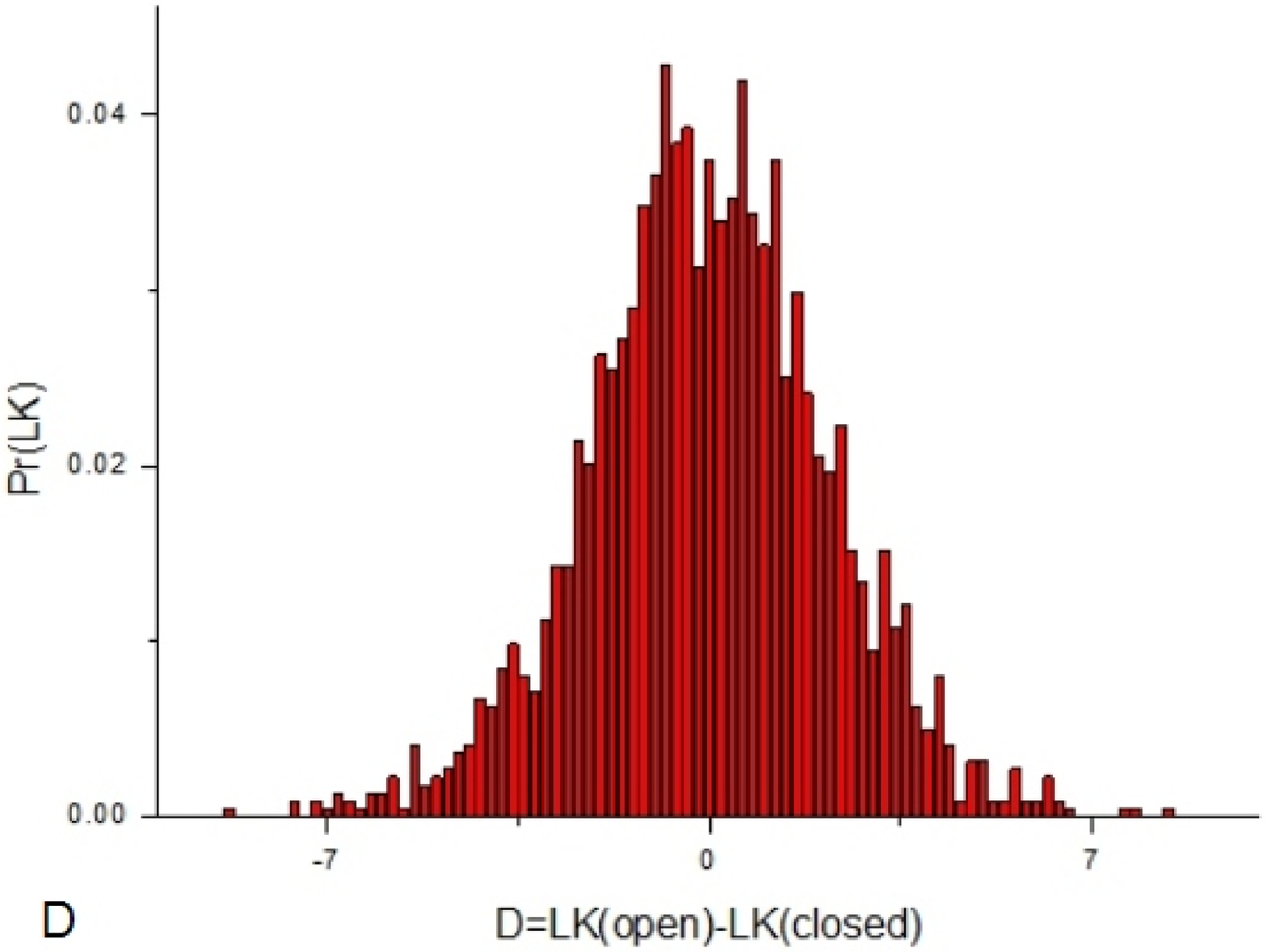}
     \caption{Figure A shows the normalized probability distribution of $LK$ for the original chains before and after end-to-end closure and Figure B the probability distribution of the corresponding differences. Figure C shows the normalized probability distribution of $LK$ for the reduced chains before and after end-to-end closure and Figure D shows the corresponding differences}
     \label{OOOC}
   \end{center}
\end{figure}

Next, we compute the local periodic linking number $LK$ after performing end-to-end closure to the original chains and compare it to the $LK$ after performing end-to-end closure to the corresponding reduced chains. Figure \ref{OCRC}C shows the normalized $LK$ for the end-to-end closures of the original and the corresponding reduced chains. We notice a smaller overlap of the two distributions. The mean absolute value of $LK$ for the original closed chains is $2.01$ and the mean absolute value of $LK$ for the reduced closed chains is $1.86$. Figure \ref{OCRC}D shows the normalized probability distribution of the difference $D=LK_{OrgClosed}-LK_{RedClosed}$. We can see that the distribution is a bit broader having standard deviation $0.92$ in comparison with the open data. The mean absolute difference is $0.57$. As we mentioned in Remark \ref{crem}, we could expect $LK$ to be invariant for closed chains in most cases. The end-to-end closure is performed in this case before and after the application of the CReTA algorithm. Thus the observed differences between the original and reduced chains are most likely due to movements during the reduction algorithm in which a chain crossed the path of the end-to-end closure.
We can see that the method of end-to-end closure retains much but not all the topological information from the original in the reduced systems.

Next we compute the difference between the $LK$ for open and closed original chains and the difference between the $LK$ for open and closed reduced chains. Figures \ref{OOOC}A and \ref{OOOC}C show the normalized $LK$ for the original open and the corresponding original closed chains and the normalized $LK$ for the reduced open and the corresponding reduced closed chains respectively. In both cases, we notice an even smaller overlap between the two distributions. In Figures \ref{OOOC}B and \ref{OOOC}D is shown the normalized probability distribution of the difference $LK_{OrgOpen}-LK_{OrgClosed}$ and of the difference $LK_{RedOpen}-LK_{RedClosed}$ respectively. We can see that the distributions are broader, with standard deviations $2.00$ and $1.91$, respectively. The mean absolute difference is $1.57$ and $1.52$ respectively. This shows that the local periodic linking number retains similar but distinct information for open and closed chains. %can detect topological differences between open and closed chains as expected.

\section{Conclusions}
The local periodic linking number, $LK$, provides the means to effectively study the linking present in a periodic system as a measure of its entanglement. By employing $LK$ and comparing the linking distributions before and after application of the CReTA algorithm we have shown that CReTA retains the linking information, and that $LK$ reflects topological differences between open and closed chains.
Our numerical results also show that the method of end-to-end closure of the chains is very sensitive to the movement of the chains, even if the TCs are preserved. This suggests that the local periodic linking number may be a stronger measure of entanglement for open chains.

%\section*{Acknowledgements}
%We would like to thank ...........


\begin{thebibliography}{99}

\bibitem{Ed} M. Doi and S. F. Edwards, \textit{The Theory of Polymer Dynamics} (Clarendon Press, Oxford, 1986)

\bibitem{PolPhys} M. Rubinstein and R. Colby, \textit{Polymer Physics}, (Oxford University Press,2003).

\bibitem{kazu} K. Iwata and S. F. Edwards, J. Chem. Phys.,\textbf{90}, (1988), 4567.

\bibitem{krog1} M. Kr\"{o}ger, Comp. Phys. Commun., \textbf{168}, (2005), 209.

%\bibitem{Lang}W. ~Michalke, M. ~Lang, S. ~Kreitmeier and D. ~Goritz \JL{\ Phys.\ Rev.\ E.,64,2001,012801}.

\bibitem{Laso} M. Laso, N. C. Karayiannis, K. Foteinopoulou, M. L. Mansfield and M. Kr\"{o}ger, Soft Matter, \textbf{5}, (2009), 1762.

\bibitem{Cossms} C. Tzoumanekas and D. Theodorou, Current Opinion in Solid State and Materials Science, \textbf{39}, (2006), 4592.


\bibitem{TopAn} C. Tzoumanekas and D. Theodorou, Macromolecules, \textbf{39}, (2006), 4592.


\bibitem{Th} P. Pant and D. Theodorou, Macromolecules, \textbf{28}, (1995), 7224.

\bibitem{Theod} D. Theodorou, \textit{Variable-Connectivity Monte Carlo Algorithms for the Atomistic Simulation of Long-Chain Polymer Systems} (P.Nielaba, M. Mareschal, G. Ciccotti, Springer-Vwelang Berlin Heidelberg, 2002)


\bibitem{Th3} M. J. Kotelyanskii and D. N. Theodorou, \textit{Simulation Methods for Polymers}, (Marcel Dekker, New York, 2004)


\bibitem{TheodPBC} D. N. Theodorou, T. D. Boone, L. R. Dodd and K. F. Mansfield, Makromol. Chem. Theory Simul., \textbf{2}, (1993), 191


\bibitem{ACN}Y. Diao, A. Dobay, R. B. Kushner, K. Millett and A. Stasiak, J. Phys. A: Math. Gen., \textbf{36}, (2003), 11561.

\bibitem{Uni}E. Panagiotou, K. C. Millett and S. Lambropoulou, J. Phys. A: Math. Theor., \textbf{43}, (2010), 045208.


\bibitem{hira} N. Hirayama, K. Tsurusaki and T. Deguchi, J. Phys. A: Math. Theor., \textbf{42}, (2009), 105001.



\bibitem{Onset12} C. Tzoumanekas, F. Lahmar, B. Rousseau and D. N. Theodorou, Macromolecules, \textbf{42}, (2009), 7474\\
F. Lahmar, C. Tzoumanekas, D. N. Theodorou and B. Rousseau, Macromolecules, \textbf{42}, 2009, 7485.


\bibitem{Ralph} R. Everaers, S, K. Sukumaran, G. S. Grest, C. Svaneborg, A. Sivasubranian, K. Kremer, Science, \textbf{303}, (2004), 823


\bibitem{Fot} K. Foteinopoulou, N. Ch. Karayiannis, M. Laso, M. Kr\"{o}ger and M. L. Mansfield, Phys. Rev. Lett., \textbf{101}, (2008), 25702.

\bibitem{hanse} E. J. Hanse van Rensburg, D. W. Sumners, E. Wasserman and S. G. Whittington, J. Phys. A: Math. Gen., \textbf{25}, (1992), 6557.


\bibitem{LinearTying} K. Millett, A. Dobay and A. Stasiak, Macromolecules, \textbf{38}, (2004), 601.\\
K. ~Millett and B. ~Sheldon, Ser. Knots and Everything, \textbf{36}, (2005), 203.


\bibitem{virnau} P. Virnau, Y. Kantor and M. Kardar, J. Am. Chem. Soc., \textbf{127}, (2005), 15102.


\end{thebibliography}
\end{document}